\documentclass[a4paper,11pt]{article}

\usepackage{amssymb, amsmath, color}

\def\C{\mathbb{C}}
\def\R{\mathbb{R}}
\def\N{\mathbb{N}}

\def\Q{\mathbb{Q}}
\def\Z{\mathbb{Z}}

\def\sg {{\rm sg}}

\newtheorem{definition}{Definition}

\newtheorem{lemma}[definition]{Lemma}
\newtheorem{proposition}[definition]{Proposition}
\newtheorem{theorem}[definition]{Theorem}
\newtheorem{corollary}[definition]{Corollary}
\newtheorem{remark}[definition]{Remark}
\newtheorem{notation}[definition]{Notation}

           {\vspace{3.3mm}
           \noindent{\bf #1}\it}%
           {\vspace{3.3mm}}

\newenvironment{proof}[1]{
  \trivlist \item[\hskip \labelsep{\it #1}]}{\hfill\mbox{$\square$}
  \endtrivlist}

\title{Zero counting for a class of univariate Pfaffian functions}
\author{Mar\'\i a Laura Barbagallo$^{\natural,\diamondsuit,}$\footnote{Partially supported by the following grants: PIP 099/11 CONICET and UBACYT 20020120100133 (2013/2016).}, Gabriela Jeronimo$^{{ \natural,\diamondsuit},*}$, Juan Sabia$^{{\flat, \diamondsuit}, *}$
\\[3mm]
{\small ${\natural}$ Departamento de Matem\'atica, FCEN, Universidad de Buenos Aires, Argentina}\\
{\small $\flat$ Departamento de Ciencias Exactas, CBC, Universidad de Buenos Aires, Argentina}\\
{\small ${\diamondsuit}$ IMAS, CONICET--UBA, Argentina}
}

\begin{document}

\maketitle

\begin{abstract}
We present a new procedure to count the number of real zeros of a class of univariate Pfaffian functions of order $1$. The procedure is based on the construction of Sturm sequences for these functions and relies on an oracle for sign determination. In the particular case of $E$-polynomials, we design an oracle-free effective algorithm solving this task within exponential complexity. In addition, we give an explicit upper bound for the absolute value of the real zeros of an $E$-polynomial.
\end{abstract}

Keywords: Pfaffian functions; zero counting; Sturm sequences; complexity.

\section{Introduction}

Pfaffian functions, introduced by Khovanskii in the late '70 (see \cite{Kho80}), are analytic functions that satisfy first order partial differential equation systems with polynomial coefficients.
A fundamental result proved by Khovanskii
(\cite{Kho91}) states that a system of  $n$ equations given by Pfaffian functions in $n$ variables defined on a domain $\Omega$ has finitely many non-degenerate solutions in $\Omega$, and this number can be bounded in terms of syntactic parameters associated to the system.

From the algorithmic viewpoint, \cite{GV04} presents a summary of quantitative and complexity  results for Pfaffian equation systems essentially based on Khovanskii's bound. The known elimination procedures in the Pfaffian structure rely on the use of an \emph{oracle} (namely, a blackbox subroutine which always gives the right answer)
to determine consistency for systems of equations and inequalities given by Pfaffian functions.
However, for some classes of Pfaffian functions the consistency problem is algorithmically decidable: for instance, an algorithm for the consistency problem of systems of the type $f_1(x) \ge 0,\dots, f_k(x)\ge 0, f_{k+1}(x)>0, \dots, f_l(x)>0$, where $x= (x_1,\dots, x_n)$, $f_i(x) = F_i(x, e^{h(x)})$ and
$F_i$ $(1\le i \le l)$ and $h$ are polynomials with integer coefficients, is given in \cite{Vor92}. This result allows the design of algorithms to solve classical related geometric problems (see, for example, \cite{RV94}).  More generally, the decidability of the theory of the real exponential field (i.e.~the theory of the structure $\R_{\mbox{exp}} = \langle \R; +, \cdot, -, 0, 1, \mbox{exp}, < \rangle$) was proved in \cite{MW96} provided Shanuel's conjecture is true.

In this paper, we design a symbolic procedure to count the exact number of zeros in a real interval of a univariate Pfaffian function of the type $f(x) = F(x,\varphi(x))$, where
$F$ is a polynomial in $\mathbb{Z}[X,Y]$ and $\varphi$ is a univariate Pfaffian function of order $1$ (see \cite[Definition 2.1]{GV04}). The procedure is based on the construction of a family of Sturm sequences associated to the given function $f(x)$, which is done by means of polynomial subresultant techniques (see, for instance, \cite{BPR}). As it is usual in the literature on the subject, we assume the existence of an oracle to determine the sign a Pfaffian function takes at a real algebraic number.
Sturm sequences in the context of transcendental functions were first used in \cite{Richardson91} to extend the cylindrical decomposition technique to non-algebraic situations. In \cite{Wolter93}, this approach was followed to count the number of real roots of exponential terms of the form $p(x) + q(x) e^{r(x)}$, where $p, q$ and $r$ are real polynomials. Later in \cite{Maignan98}, the same technique is applied to treat the case of functions of the type $F(x,e^x)$, where $F$ is an integer polynomial.

A function of the form
$$f(x) = F(x, e^{h(x)}),$$
where $F$ and $h$ are polynomials with real coefficients, is called an $E$-polynomial (\cite{Vor92}). For these particular functions, we give an effective symbolic algorithm solving the zero-counting problem with no calls to oracles. To this end, we construct a subroutine to determine the sign of univariate $E$-polynomials at real algebraic numbers. Our algorithms only perform arithmetic operations and comparisons between rational numbers. In order to deal with real algebraic numbers, we represent them by means of their Thom encodings (see Section \ref{sec:algorithms}). The main result of the paper is the following:

\begin{theorem}\label{mainThm}
Let $f(x) = F(x, e^{h(x)})$ be an $E$-polynomial defined by polynomials $F\in \Z[X,Y]$ and $h\in \Z[X]$ with degrees bounded by $d$ and coefficients of absolute value at most $H$, and let $I=[a, b]$ be a closed interval or $I= \mathbb{R}$. There is an algorithm that computes the number of zeros of $f$ in $I$ within complexity $(2dH)^{d^{O(1)}}$.
\end{theorem}

Finally, we prove an explicit upper bound for the absolute value of the real zeros of an $E$-polynomial
in terms of the degrees and absolute values of the coefficients of the polynomials involved.
This bound could be used to separate and approximate the real zeros of an $E$-polynomial. It provides an answer to the `problem of the last root' for this type of functions. Previously, in \cite{Wolter85}, the existence of such a bound was established for general exponential terms, but even though it is given by an inductive argument with a computable number of iterations, the bound is not explicit. Algorithms for the computation of upper bounds for the real roots of functions of the type $P(x, e^x)$ or, more generally, $P(x, \mbox{trans}(x))$, with $P$ an integer polynomial and $\mbox{trans}(x)= e^x, \ \ln(x)$ or $\arctan(x)$ are given in \cite{Maignan98} and \cite{McCallumWeisp2012} respectively.

The paper is organized as follows: in Section \ref{sec:preliminaries}, we fix the notation and recall some basic theoretical and algorithmic results on univariate polynomials. Section
\ref{sec:Sturm} is devoted to the construction of Sturm sequences for the Pfaffian functions we deal with. In Section \ref{sec:generalalgorithm}, we present our general procedure for zero counting. Finally, in Section  \ref{sec:Epolynomials}, we describe the algorithms and prove our main results on $E$-polynomials.

\section{Preliminaries}\label{sec:preliminaries}

\subsection{Basic notation and results}\label{subsec:not}

Throughout the paper, we will deal with univariate and bivariate polynomials. For a polynomial
$F\in \Z[X,Y]$, we write $\deg_X(F)$ and $\deg_Y(F)$ for the degrees of $F$ in the variables $X$ and $Y$ respectively, $H(F)$ for its height, that is, the maximum of the absolute values of its coefficients in $\mathbb{Z}$, and $\textrm{cont}(F)\in \Z[X]$ for the gcd of the coefficients of $F$ as a polynomial in $\Z[X][Y]$.

Note that, if $p_1, p_2\in \Z[X]$ are polynomials with degrees bounded by $d_1$ and $d_2$, and heights bounded
by $H_1$ and $H_2$, then $H(p_1p_2) \le (\min\{d_1, d_2\} +1) H_1 H_2$.

If $f$ is a real univariate analytic function, we denote its derivative by $f'$ and, for $k>1$, its $k$th successive derivative by $f^{(k)}$.

For $\gamma= (\gamma_0,\dots, \gamma_N)\in \R^{N+1}$ with $\gamma_i \ne 0$ for every $0\le i \le N$, the \emph{number of variations in sign} of $\gamma$ is the cardinality of the set $\{1\le i\le N : \gamma_{i-1} \gamma_i <0\}$. For a tuple $\gamma$ of arbitrary real numbers, the number of variations in sign of $\gamma$ is defined as the number of variations in sign of the tuple which is obtained from $\gamma$ by removing its zero coordinates. Given $x\in \mathbb{R}$ and a sequence of univariate real functions $\mathbf{f} =(f_0,\dots, f_N)$ defined at $x$,  we write $v(\mathbf{f},x)$ for the number of variations in sign of the $(N+1)-$tuple $(f_0(x),\ldots,f_N(x))$.

\bigskip

We recall some well-known bounds on the size of roots of univariate polynomials (see \cite[Proposition 2.5.9 and Theorem 2.5.11]{MS}).

\begin{lemma} \label{sizeofroots}
  Let $p= \sum_{j=0}^d a_j X^j\in \C[X]$, $a_d \ne 0$.
Let $r(p) := \max\{ |z|: z\in \C, \ p(z) =0\}$. Then:
\begin{enumerate}
\item[i)] $r(p) < 1 + \max_{} \left\{\left| \dfrac{a_j}{a_d}\right|: 0\le j \le d-1 \right\}$
\item[ii)] $r(p) <  \left( 1+\displaystyle \sum_{0\le j\le d-1} \left| \dfrac{a_j}{a_d}\right|^2\right)^{1/2}$
\end{enumerate}
\end{lemma}

We will also use the following lower bound for the separation of the roots of a univariate polynomial with integer coefficients (see \cite[Theorem 2.7.2]{MS}):

 \begin{lemma}\label{separation}
 Let $p\in \Z[X]$ be a polynomial of degree $d\ge 2$, and $\alpha_1,\dots, \alpha_d$ be all the roots of $p$. Then
\[ \min \{| \alpha_i-\alpha_j | : \alpha_i\ne \alpha_j \} > d^{-\frac{d+2}{2}} (d+1)^{\frac{1-d}{2}} H(p)^{1-d}.\]
\end{lemma}

\bigskip

A basic tool for our results is the well-known theory of subresultants for univariate
polynomials with coefficients in a ring and its relation with polynomial remainder sequences (see \cite[Chapter 8]{BPR}).

Let $F(X,Y)$ and $G(X,Y)$ be polynomials in $\Z[X,Y]$ of degrees $d$ and $e$ in the variable $Y$ respectively. Assume $e<d$. Following \cite[Notation 8.33]{BPR}, for every  $-1\le j \le d$, let ${\rm SRes}_j$ be the $j$th signed subresultant of $F$ and $G$ considered as polynomials in $\Z[X][Y]$.
By the structure theorem for subresultants (see \cite[Theorem 8.34 and Proposition 8.40]{BPR}),
we have that $${\rm SRes}_{e-1} =  - \textrm{Remainder}( (-1)^{(d-e-1)(d-e)/2} {\rm lc}(G)^{d-e+1} F, G),$$
where ${\rm lc}(G)$ is the leading coefficient of $G$ and, for an index $i$ with $1\le i \le d$ such that ${\rm SRes}_{i-1}$ is non-zero of degree $j$:
\begin{itemize}
\item If ${\rm SRes}_{j-1} =0$, then ${\rm SRes}_{i-1} = \textrm{gcd}(F,G)$ up to a factor in $\Z[X]$.
\item If ${\rm SRes}_{j-1} \ne0$ has degree $k$,
$$s_j t_{i-1} {\rm SRes}_{k-1} = - \textrm{Remainder}(s_k t_{j-1} {\rm SRes}_{i-1} , {\rm SRes}_{j-1})$$
and the quotient lies in $\Z[X][Y]$. Here, $s_l$ denotes the $l$th subresultant coefficient of $F$ and $G$ as defined in \cite[Notation 4.22]{BPR} and $t_l$ is the leading coefficient of ${\rm SRes}_l$.
\end{itemize}

We define a sequence of integers as follows:
\begin{itemize}
\item $n_0= d+1$, $n_1 =d$.
\item For $i\ge 1$, if ${\rm SRes}_{n_i-1} \ne 0$, then $n_{i+1}= \deg({\rm SRes}_{n_{i}-1})$.
\end{itemize}

The polynomials
$$
R_i:= {\rm SRes}_{n_i-1}
$$
are proportional to the polynomials in the Euclidean remainder sequence associated to $F$ and $G$. Moreover, the following relations hold:
\begin{equation}\label{euclid1}
(-1)^{(d-e)(d-e+1)/2} {\rm lc}(G)^{d-e+1} R_0 =  R_1 C_1 -R_2
\end{equation}
\begin{equation}\label{euclid2}
 s_{n_{i+2}} t_{n_{i+1}-1} R_i = R_{i+1} C_{i+1} - s_{n_{i+1}} t_{n_i-1} R_{i+2}\qquad \hbox{for } i\ge 1
 \end{equation}
where $C_i\in \Z[X][Y]$ for every $i$.

\subsection{Algorithms and complexity}\label{sec:algorithms}

The algorithms we consider in this paper are described by arithmetic networks over
$\Q$ (see \cite{vzG86}). The notion of complexity of an algorithm we
consider is the number of operations and comparisons in $\Q$.
The objects we deal with are polynomials with coefficients in $\Q$, which are represented by the array of all their coefficients in a pre-fixed order of their monomials.

To estimate complexities we will use the following results (see \cite{vzGG}). The product of two polynomials in $\Q[X]$ of degrees bounded by $d$ can be done within complexity $O(M(d))$, where $M(d) = d \log(d) \log\log(d)$. Interpolation of a degree $d$ polynomial in $\Q[X]$ requires $O(M(d)\log(d))$ arithmetic operations.
We will use the Extended Euclidean Algorithm to compute the gcd of two polynomials in $\Q[X]$ of degrees
bounded by $d$ within  complexity $O(M(d) \log(d))$.
We will compute subresultants by means of matrix determinants, which enables us to control both the complexity and output size (an alternative method for the computation of subresultants, based on the Euclidean algorithm, can be found in \cite[Algorithm 8.21]{BPR}). For a matrix in $\Q^{n\times n}$, its determinant can be obtained within complexity $O(n^\omega)$, where $\omega<2.376$ (see \cite[Chapter 12]{vzGG}).

For a polynomial in $\Z[X]$, we will need to approximate its real roots by rational numbers and to isolate them in disjoint intervals of pre-fixed length with rational endpoints. There are several known algorithms achieving these tasks (see, for instance, \cite{SagraloffMelhorn2016} and the references therein). Here we use a classical approach via Sturm sequences. The complexity of the algorithm based on this approach is suboptimal. However, the complexity order of the procedures in which we use it as a subroutine would not change even if we replaced it with the one with the best known complexity bound.

\begin{lemma}\label{intervalsforroots}
Let $p\in \Z[X]$ be a polynomial of degree bounded by $d$ and $\epsilon \in \Q$,
$\epsilon >0$. There is an algorithm which computes finitely many pairwise disjoint intervals $I_j= (a_j, b_j]$ with $a_j, b_j \in \Q$ and $b_j-a_j\le \epsilon$ such that each $I_j$ contains at least one real root of $p$ and every real root of $p$ lies in some $I_j$. The complexity of the algorithm is of order $O(d^3 \log(H(p)/\epsilon))$.
\end{lemma}

\begin{proof}{Proof.} The algorithm works recursively.
Starting with the interval $J= (-(1+H(p)), 1+H(p)]$, which contains all the real roots of $p$ (see Lemma \ref{sizeofroots}), at each intermediate step, finitely many intervals are considered. Given an interval $J=(a, b]$ with $\{ p=0\} \cap J\ne \emptyset$ and $|J|>\epsilon$, the procedure runs as follows:
\begin{itemize}
\item Let $c= \frac{a+b}{2}$ and $J_r = (c, b]$.
\item If $p(c) \ne 0$, let $J_l=(a, c]$.
\item If $p(c) =0$ and $c-\epsilon >a$, let $I= (c-\epsilon, c]$ and $J_l = (a, c-\epsilon]$.
If $p(c)=0$ and $c-\epsilon\le a$, take $I=(a, c]$. (Note that, in any case, $I$ contains a real root of $p$ and has length at most $\epsilon$.)
\item Determine, for each of the intervals $J_r$ and $J_l$, whether $p$ has a real root in that interval or not. Keep the intervals that contain real roots of $p$.
\end{itemize}
The recursion finishes when the length of all the intervals is at most $\epsilon$. The output consists of all the intervals of length at most $\epsilon$ containing roots of $p$, including the intervals $I$ appearing at intermediate steps.

In order to determine whether $p$ has a real root in a given interval, we use the Sturm sequence of $p$ and $p'$ (see \cite[Theorem 2.50]{BPR}), which is computed within complexity $O(M(d) \log(d))$ by means of the Euclidean Algorithm.

At each step of the recursion, we keep at most $d$ intervals together with the number of variations in sign of the Sturm sequence evaluated at each of their endpoints. For each of these intervals, the procedure above requires at most $2d+1$ additional evaluations of polynomials of degrees at most $d$. Then, the complexity of each recursive step is of order $O(d^3)$.

Since the length of the intervals at the $k$th step is at most $\frac{1+H(p)}{2^{k-1}}$, the number of steps is at most $1+\lceil \log( \frac{1+H(p)}{\epsilon}) \rceil$.  Therefore, the overall complexity is $O(d^3 \log(H(p)/\epsilon))$.
\end{proof}

In order to deal with real algebraic numbers in a symbolic way, we will use
\emph{Thom encodings}.
We recall here their definition and main properties (see \cite[Chapter 2]{BPR}).
Given $p \in \R[X]$ and a real root $\alpha$ of $p$, the Thom encoding of $\alpha$ as a root of $p$ is
the sequence $(\text{sign}(p'(\alpha)), \dots, \text{sign}(p^{(\deg p)}(\alpha))),$ where we represent the sign with an element of the set $\{  0, 1,-1\}$.
Two different real roots of $p$ have different Thom encodings. In addition, given
the Thom encodings of two different real roots $\alpha_1$ and $\alpha_2$ of $p$, it is possible to decide which is the smallest between $\alpha_1$ and $\alpha_2$ (see \cite[Proposition 2.28]{BPR}).

For a polynomial $p\in \R[X]$, we will denote
$$\text{Der}(p):=(p, p' ,\dots, p^{(\deg p)})$$

A useful tool to compute Thom encodings and manipulate real algebraic numbers is an effective procedure for the determination of feasible sign conditions on real univariate polynomials.
For $p_1,\dots, p_s \in \R[X]$,  a \emph{feasible sign condition for $p_1,\dots, p_s$ on a finite set $Z\subset \R$} is an $s$-tuple $(\sigma_1,\dots, \sigma_s) \in \{=, >, <\}^s$  such that $\{ x\in Z : p_1(x) \sigma_1 0,\dots, p_s(x) \sigma_s 0\} \ne \emptyset$.

\begin{lemma}\label{perrucci11} (see \cite[Corollary 2]{Perrucci11})
Given $p_0, p_1,\dots, p_s \in \R[X]$, $p_0\not\equiv 0$, $\deg p_i \le d$ for $i = 0,\dots, s$, the feasible sign conditions for $p_1,\dots, p_s$ on $\{p_0 = 0\}$ can be computed algorithmically within $O(sd^2 \log^3( d))$
operations. Moreover, if $p_0$ has $m$ roots in $\R$, this can be done within $O(smd \log(m) \log^2(d))$ operations. The output of the algorithm is a list of $s$-tuples in $\{0,1,-1\}^s$, where $0$ stands for $=$, $1$ for $>$ and $-1$ for $<$.
\end{lemma}

\section{Sturm sequences and zero counting for Pfaffian functions}\label{sec:Sturm}

Following \cite{Heindel71}, we introduce the notion of a Sturm sequence for a continuous function in a real interval:

\begin{definition} \label{defisturm}
Let $f_0:(a,b)\rightarrow\R$ be a continuous function of a single variable. A sequence of continuous functions $\mathbf{f} =(f_0,\ldots,f_N)$ on $(a, b)$ is said to be a \emph{Sturm sequence for $f_0$ in the interval $(a,b)$ } if the following conditions hold:
\begin{enumerate}
	\item If $f_0(y)=0$, there exists $\epsilon>0$ such that $f_1(x)\neq0$ for every $x\in (y-\epsilon,y+\epsilon)\subseteq(a,b)$, $x\neq y$, $f_0(x)f_1(x)<0$ for $y-\epsilon<x<y$ and  $f_0(x)f_1(x)>0$ if $y<x<y+\epsilon$.
	\item For every $i=1,\ldots,N-1$, if $f_i(x)=0$ for $x\in(a,b)$, then $f_{i-1}(x)f_{i+1}(x)<0$.
	\item $f_N(x)\neq 0$ for every $x\in(a,b)$.
\end{enumerate}	
\end{definition}

Recalling that, for a given $x\in \R$, $v(\mathbf{f}, x)$ denotes the number of variations in sign of the $(N+1)$-tuple $(f_0(x), \dots, f_N(x))$, we have the following analog of the classical Sturm theorem:

\begin{theorem}(\cite[Theorem 2.1]{Heindel71}) \label{resultadosturm}
Let $f_0:(a,b)\rightarrow\R$ be a continuous function of a single variable. Let $\mathbf{f}=(f_0,\ldots,f_N)$ be a Sturm sequence for $f_0$ in the interval $(a,b)$ and let $a<c< d<b$. Then, the number of
distinct real zeros of $f_0$ in the interval $(c,d]$ is $v(\mathbf{f},c)-v(\mathbf{f},d)$.
\end{theorem}

The aim of this section is to build Sturm sequences for a particular class of Pfaffian functions we introduce below. For the definition of Pfaffian functions in full generality and the basic properties of these functions see, for instance, \cite{GV04}.

Given a polynomial $\Phi\in \Z[X,Y]$ with $\deg_Y(\Phi)>0$,
let $\varphi$ be a function satisfying the differential equation
\begin{equation}\label{defvarphi}
\varphi'(x) = \Phi(x, \varphi(x)).
\end{equation}
Note that $\varphi$ is analytic on its domain, which may be a proper subset of $\R$.

We are going to work with Pfaffian functions of the type
\[ f(x) = F(x, \varphi(x)),\]
where $F\in \Z[X,Y]$. 

Taking into account that the first derivative of such a function is
\[ \frac{\partial F}{\partial X}(x, \varphi(x)) + \frac{\partial F}{\partial Y}(x, \varphi(x)). \Phi(x, \varphi(x)),\]
we define, for any $F\in \Z[X, Y]$, the polynomial $\widetilde F\in \Z[X,Y]$ (associated with $\Phi$)  as follows:
\begin{equation}\label{Ptilde}
\widetilde F(X,Y) = \dfrac{\partial F}{\partial X}(X,Y) + \dfrac{\partial F}{\partial Y}(X,Y)  \Phi(X,Y).
\end{equation}
Thus, we have that
\[ f'(x) = \widetilde F(x, \varphi(x)).\]

Due to the following result, in order to count the number of real zeros of a function $f$ as above,  we will assume from now on, without loss of generality, that ${\rm Res}_Y(F, \widetilde F) \ne 0$.

\begin{lemma}\label{lem:resP} Let $\Phi, \varphi$ be as in equation (\ref{defvarphi}) and let $F\in \Z[X, Y]$ with $\deg_Y(F) >0$.
  There exists a polynomial $P\in \Z[X, Y]$ such that ${\rm Res}_Y(P, \widetilde P) \ne 0$ and $P(x, \varphi(x))$ has the same real zeros as $F(x, \varphi(x))$. Moreover, the polynomial $P$ can be effectively computed from $F$ and $\Phi$.
 \end{lemma}

\begin{proof}{Proof.} Without loss of generality, we may assume that $F$ is square-free. Suppose that ${\rm Res}_Y(F, \widetilde F)= 0$. Write $F = \textrm{cont}(F) \, F_0$. Then, ${\rm Res}_Y(F_0, \widetilde F_0)= 0$ and so, the greatest common divisor of $F_0$ and $\widetilde F_0$ is a polynomial $S \in \Z[X, Y]$ of positive degree in $Y$. If
$$F_0= S \, U \quad \hbox{ and } \quad \widetilde F_0 = S \,  V$$ for $U, V \in \Z[X, Y]$, we have that
\[ f_0(x) = F_0(x,\varphi(x))= S(x, \varphi(x)) \, U(x, \varphi(x)) \ \hbox{ and } \ f_0'(x) = \widetilde F_0(x,\varphi(x)) = S(x, \varphi(x)) \, V(x, \varphi(x)),\]
which implies that a zero $\xi$ of $f_0$ which is not a zero of $ U(x, \varphi(x))$ satisfies that ${\rm mult}(\xi, f_0) = {\rm mult}(\xi, S(x, \varphi(x))) \le {\rm mult}(\xi, f_0')$, leading to a contradiction. Then, $f_0$ and $U(x, \varphi(x))$ have the same zero set in $\mathbb{R}$.
As
$$\widetilde F_0 = \widetilde{(S \, U)} = \widetilde S \, U + S \, \widetilde U,$$
it follows that, if $T\in \Z[X,Y]$ is a common factor of $U$ and $\widetilde U$ with positive degree in $Y$, then $T$ divides $\widetilde F_0= S\,V$. Since $U$ and $V$ are relatively prime polynomials, then $T$ divides $S$ and, therefore $T^2 $ divides $F_0$, contradicting the fact that $F_0$ is square-free.

The lemma follows considering the polynomial $P= \textrm{cont}(F)\, U$.
\end{proof}

We will apply the theory of subresultants introduced in Section \ref{sec:preliminaries} in order to get Sturm sequences for $f$.

Let $$F_1 = {\rm Remainder}({\rm lc}(F)^{D} \widetilde F, F)\in \Z[X][Y],$$ where $D$ is the smallest even integer greater than or equal to $ 1+ \deg_Y(\widetilde F) - \deg_Y(F)$.

\begin{notation}\label{not:subres}
Following Section \ref{subsec:not}, for $i=0,\dots,N$, let $R_i:= {\rm SRes}_{n_i-1} \in \Z[X][Y]$  be the $(n_i-1)$th subresultant polynomial associated to $F$ and $F_1$, $\tau_i:=t_{n_i-1}\in \Z[X]$ be the leading coefficient of $R_i$ and, for $i=2,\dots, N+1$, let $\rho_i:=s_{n_i}\in \Z[X]$ be the $n_i$th subresultant coefficient of $F$ and $F_1$.
\end{notation}

\begin{definition}\label{def:signseq}
For an interval $I =(a,b)$  containing no root of the polynomials $\tau_i$ for $i=0,\dots, N$ or $\rho_i$ for $i=2,\dots, N+1$, we define inductively a sequence $(\sigma_{I,i})_{0\le i \le N}\in \{1,-1\}^{N+1}$  as follows:
\begin{itemize}
\item $\sigma_{I,0}=\sigma_{I,1}=1$,
\item $\sigma_{I,2} = (-1)^{\frac{1}{2}(\deg_Y(F)-\deg_Y(F_1))(\deg_Y(F)-\deg_Y(F_1)+1)} \textrm{sg}_I(\textrm{lc}(F_1))^{\deg_Y(F)-\deg_Y(F_1)+1}$,
\item $\sigma_{I,i+2}=  \textrm{sg}_I(\rho_{i+2} \tau_{i+1} \rho_{i+1} \tau_i) \sigma_{I, i}$,
    \end{itemize}
where, for a continuous function $g$ of a single variable with no zeros in $I$, $\textrm{sg}_I(g)$ denotes the (constant) sign of $g$ in $I$.
For $i=0, \dots, N$, we define $$F_{I,i} = \sigma_{I,i} R_i\in \Z[X,Y].$$
Finally, if $I$ is contained in the domain of $\varphi$, we introduce the sequence of Pfaffian functions $\mathbf{f}_I= (f_{I,i})_{0\le i\le N}$ defined by $$f_{I,i}(x) =F_{I,i}(x, \varphi(x)).$$
\end{definition}

\begin{proposition} \label{prop:sturmI} Let $F\in \Z[X,Y]$, $\deg_Y (F)>0$, and let $\varphi$ be a Pfaffian function satisfying  $\varphi'(x) = \Phi(x, \varphi(x))$, where $\Phi \in \Z[X,Y]$ with $\deg_Y(\Phi)>0$.
Consider the function $f(x) = F(x, \varphi(x))$. Let $\widetilde F\in \Z[X, Y]$ be defined as in (\ref{Ptilde}). Assume that the resultant ${\rm{Res}}_Y(F,\widetilde F)\in \Z[X]$ is not zero.
With the notation and assumptions of Definition \ref{def:signseq}, the sequence of Pfaffian functions $\mathbf{f}_I= (f_{I,i})_{0\le i\le N}$  is a Sturm sequence for $f$ in $I=(a,b)$.
\end{proposition}

\begin{proof}{Proof.}
For simplicity, as the interval $I$ is fixed, the subindex $I$ will be omitted throughout the proof.

First we prove that $f_0$ and $f_1$ do not have common zeros in $I$. Suppose $\alpha\in I$  is a common zero of $f_0$ and $f_1$. Then $F(\alpha, \varphi(\alpha)) = 0$ and $F_1(\alpha, \varphi(\alpha)) = 0$; therefore, $\rho_{N+1}(\alpha) = \textrm{Res}_Y(F,F_1)(\alpha) = 0$, contradicting the assumptions on $I$.

From this fact, taking into account that $f_0 = f$, and $f_1$ has the same sign as $f'$ at any zero of $f$ lying in $I$, condition 1 of Definition \ref{defisturm} follows.

To prove that condition 2 holds, first note that if $f_j(\alpha) = 0$ and $f_{j+1}(\alpha)=0$
for some $\alpha\in I$, since $\rho_i$ and $\tau_i$ do not have zeros in $I$, by identities (\ref{euclid1}) and (\ref{euclid2}), $\alpha$ is a common zero of all $f_i$s, contradicting the fact that $f_0$ and $f_1$ do not have common zeros in $I$. Then, condition 2 in Definition \ref{defisturm} follows from the definition of the signs $\sigma_i$ and identities (\ref{euclid1}) and (\ref{euclid2}).

Condition 3 follows from the assumption that $\tau_{N}$, which equals $f_N$ up to a sign, does not have zeros in $I$.
\end{proof}

In order to count the number of zeros of a Pfaffian function in an open interval, provided that the function is defined in its endpoints, we introduce the following:

\begin{notation}
Let $f: J \to \mathbb{R}$ be a non-zero analytic function defined in an open interval $J\subset \mathbb{R}$ and let $c\in J$. We denote
\[ \sg(f, c^+) = \begin{cases} {\rm sign} (f(c)) & {\rm if } \ f(c) \ne 0\\
{\rm sign} (f^{(r)}(c)) &
{\rm if }\  {\rm{mult}}(c, f) = r
\end{cases}\]
and
\[ \sg(f, c^-) = \begin{cases} {\rm sign} (f(c)) & {\rm if } \ f(c) \ne 0\\
{\rm sign} ((-1)^{r} f^{(r)}(c)) &
{\rm if }\  {\rm{mult}}(c, f) = r
\end{cases}\]
where ${\rm{mult}}(c, f)$ is the multiplicity of $c$ as a zero of $f$.

For a sequence of non-zero analytic functions $\mathbf{f}= (f_0,\dots, f_N)$ defined in $J$, we write $v(\mathbf{f}, c^+)$ for the number of variations in sign  in $(\sg(f_0, c^+), \dots, \sg(f_N, c^+))$ and $v(\mathbf{f}, c^-)$ for the number of variations in sign in $(\sg(f_0, c^-), \dots, \sg(f_N, c^-))$.

\end{notation}

Note that $\sg(f, c^+)$ is the sign that $f$ takes in $(c, c+\varepsilon)$  and $\sg(f, c^-)$ is the sign that $f$ takes in $(c-\varepsilon, c)$ for a sufficiently small $\varepsilon>0$. Then, by Theorem \ref{resultadosturm}, we have:

\begin{proposition}
With the assumptions and notation of Proposition \ref{prop:sturmI}, if, in addition, the closed interval $[a,b]$ is contained in the domain of $\varphi$, the number of zeros of the function $f$ in the open interval $I=(a, b)$ equals
$v(\mathbf{f}_I, a^+) - v(\mathbf{f}_I, b^-)$.
\end{proposition}

As a consequence, we get a formula for the number of zeros of the Pfaffian function $f$ in any bounded interval:

\begin{theorem}
Let $f(x) = F(x, \varphi(x))$, where  $F\in \Z[X,Y]$, $\deg_Y (F)>0$,  and $\varphi$ is a Pfaffian function satisfying $\varphi'(x) = \Phi(x, \varphi(x))$ for a polynomial $\Phi \in \Z[X,Y]$ with $\deg_Y(\Phi)>0$. Assume ${\rm Res}_Y(F,\widetilde F)\ne 0$. Consider a bounded open interval $ (\alpha, \beta) \subset \mathbb{R}$ such that $[\alpha, \beta]$ is contained in the domain of $\varphi$.

Let $\rho_i$ and $\tau_i$ be the polynomials in $\Z[X]$ introduced in Notation \ref{not:subres}. If $\alpha_1<\alpha_2<\dots< \alpha_k$ are all the roots in $(\alpha, \beta)$  of  $\rho_i$ and $\tau_i$,  the number of zeros of $f$ in $[\alpha, \beta] $ equals
\[ \# \{ 0\le j\le k+1 : f(\alpha_j) = 0\} + \sum_{j=0}^k v(\mathbf{f}_{I_j}, \alpha_j^+) - v(\mathbf{f}_{I_j}, \alpha_{j+1}^-), \]
where $\alpha_0 = \alpha$, $\alpha_{k+1} = \beta$ and, for every $0\le j\le k$, $I_j = (\alpha_j, \alpha_{j+1})$ and $\mathbf{f}_{I_j}$ is the sequence of functions introduced in Definition \ref{def:signseq}.
\end{theorem}

\section{Algorithmic approach}\label{sec:generalalgorithm}

Let  $\varphi$ be a Pfaffian function satisfying  $$\varphi'(x) = \Phi(x, \varphi(x))$$ for a polynomial $\Phi \in \Z[X,Y]$. Let $\delta_Y:=\deg_Y(\Phi)>0$ and $\delta_X:= \deg_X(\Phi)$.

In this section, we describe an algorithm for counting the number of zeros in a bounded interval contained in the domain of $\varphi$ of a function of the type
$$f(x) = F(x, \varphi(x)),$$
where $F\in \Z[X, Y]$ with $\deg_Y(F)>0$.

To estimate the complexity of the algorithm, we need an upper bound for the multiplicity of a zero of a function of this type. Here, we present a bound in our particular setting which takes into account the degrees in each of the variables $X$ and $Y$ of the polynomials involved in the definition of the functions.  A general upper bound on the multiplicity of Pfaffian intersections depending on the \emph{total} degrees of the polynomials can be found in \cite[Theorem 4.3]{GV04}. Even though both bounds are of the same order, our bound may be smaller when the total degrees are greater than the degrees with respect to each variable.

\begin{lemma}\label{multiplicity} With the previous notation, let $g(x) = G(x, \varphi(x))$ with $G\in \Z[X, Y]$ be a nonzero Pfaffian function. For every $\alpha\in \mathbb{R}$ such that $g(\alpha)=0$,  we have $${\rm{mult}}(\alpha, g)\le 2 \deg_X(G) \deg_Y(G) + \deg_X(G) (\delta_Y-1) + (\delta_X+1) \deg_Y(G).$$
\end{lemma}

\begin{proof}{Proof.}
Assume first that $G$ is irreducible in $\mathbb{Z}[X,Y]$. If $g(\alpha)=0$, then $\text{mult}(\alpha, g) > \text{mult} (\alpha, g')$. As $g'(x) = \widetilde G(x,\varphi(x))$, then $G$ does not divide $\widetilde G$ and, therefore, $R:=\text{Res}_Y(G, \widetilde G) \ne 0$. Let $S, T \in \mathbb{Z}[X,Y]$ be such that $R = S G +  T \widetilde G$. We have that
\[ R(x) = S(x,\varphi(x)) .\, g(x) +  T(x, \varphi(x)) .\, g'(x).\]
If $\alpha$ is a multiple root of $g$, the previous identity implies that $\text{mult}(\alpha, g) \le \text{mult}(\alpha, R) +1 \le \deg(R)+1$.
Taking into account that $\deg(R) \le \deg_X(G) \deg_Y(\widetilde G) + \deg_X(\widetilde G) \deg_Y(G)$, $\deg_X(\widetilde G ) \le \deg_X(G)+\delta_X$ and $\deg_Y (\widetilde G) \le \deg_Y(G)-1+\delta_Y$, we conclude that
\[\text{mult}(\alpha, g)\le 2 \deg_X(G) \deg_Y(G) + \deg_X(G) (\delta_Y-1) + \delta_X \deg_Y(G)+1.\]

In the general case, write $G = c(X) \prod_{1\le i \le t} G_i(X,Y)^{m_i}$, where $c(X)= \text{cont}(G)$ and $G_1,\dots, G_t\in \mathbb{Z}[X,Y]$ are irreducible polynomials. For every $i$, let $g_i(x)= G_i(x, \varphi(x))$. From the previous bound, we deduce
\[ \text{mult}(\alpha, g) = \text{mult}(\alpha, c) + \sum_{1\le i \le t} m_i \, \text{mult}(\alpha, g_i) \le  \]
\[\le \deg_X(c) + \sum_{1\le i \le t} m_i \left( 2 \deg_X(G_i) \deg_Y(G_i) + \deg_X(G_i) (\delta_Y-1) + \delta_X \deg_Y(G_i)+1 \right)\]
\[ \le 2 \deg_X(G) \deg_Y(G) + \deg_X(G) (\delta_Y-1) + (\delta_X+1) \deg_Y(G).\]
\end{proof}

The theoretical results in the previous section enable us to construct the following algorithm for zero counting for a function $f(x)= F(x, \varphi(x))$, where $F\in \Z[X,Y]$. By Lemma \ref{lem:resP}, we will assume that $\text{Res}_Y(F, \widetilde F) \ne 0$.

\bigskip

\noindent \hrulefill
\smallskip

\noindent\textbf{Algorithm \texttt{ZeroCounting}}

\medskip

\noindent INPUT: A function $\varphi$ satisfying a differential equation $\varphi'(x) = \Phi(x, \varphi(x))$, a polynomial $F\in \Z[X,Y]$ such that $\text{Res}_Y(F, \widetilde F) \ne 0$, and a closed interval $[\alpha,\beta]\subset\text{Dom}(\varphi)$ with $\alpha, \beta \in \Q$.

\smallskip
\noindent OUTPUT: The number of zeros of $f(x) = F(x, \varphi(x))$ in $[\alpha,\beta]$.

\begin{enumerate}
\item Let $F_1(X,Y):= \begin{cases}
\widetilde F(X,Y) & \text{ if } \deg_Y(\widetilde F) < \deg_Y(F)\\
\text{Remainder}({\rm lc}(F)^D \widetilde F, F) & \text{ otherwise }
\end{cases}$, where $D$ is the smallest even integer greater than or equal to $1 +\deg_Y(\widetilde F) - \deg_Y(F)$.

\item Compute the polynomials $R_i$ and $\tau_i$, for $0\le i \le N$, and $\rho_i$, for $2\le i \le N+1$,  associated to $F$ and $F_1$ as in Notation \ref{not:subres}.

\item Determine and order all the real roots $\alpha_1<\alpha_2<\cdots <\alpha_k$ lying in the interval $(a, b)$ of the polynomials $\tau_i$, for $0\le i \le N$, and $\rho_i$, for $2\le i \le N+1$.

\item For every $0\le j\le k$, compute the Sturm sequence $\mathbf{f}_{I_j}= (f_{I_j, i})_{0\le i \le N}$ for $f$ in $I_j= (\alpha_j, \alpha_{j+1})$ as in Definition \ref{def:signseq}, where $\alpha_0 = \alpha $ and $\alpha_{k+1} = \beta$.

\item Decide whether $f(\alpha_j) =0$ for every $0\le j\le k+1$ and count the number of zeros.

\item For every $0\le j \le k$, compute $v_j:=v(\mathbf{f}_{I_j}, \alpha_j^+) - v(\mathbf{f}_{I_j}, \alpha_{j+1}^{-})$.

\item Compute $\# \{ 0\le j\le k+1: f(\alpha_j) = 0\} +\sum\limits_{j=1}^k v_j$.

\end{enumerate}

\noindent \hrulefill

\bigskip

\emph{Complexity analysis:}

\smallskip

Let $d_X:= \deg_X(F)$, $d_Y:=\deg_Y(F)$ and, as before, $\delta_X:= \deg_X(\Phi)$, $\delta_Y:= \deg_Y(\Phi)$.

\begin{description}

\item[Step 1.] Note that $\deg_Y(F_1) < d_Y$. In the case when $\deg_Y(\widetilde F) \ge d_Y$, in order to bound $\deg_X(F_1)$, notice that $\deg_X({\rm lc}(F)^D\widetilde F) \le D \deg({\rm lc}(F)) + d_X + \delta_X$. Then, the polynomial $F_1$ can be obtained by means of at most $D$ successive steps, each consisting of subtracting a multiple of $F$ with degree in $X$ bounded by $(D-i) \deg_X({\rm lc}(F)) + (i+1) d_X + \delta_X$ from a polynomial whose degree in $X$ is bounded by $(D-i+1) \deg_X({\rm lc}(F)) + i\, d_X + \delta_X$. Then, $\deg_X(F_1) \le (D+1) d_X +\delta_X\le (\delta_Y+2)d_X +\delta_X$.

In order to perform the computations (as polynomials in the variable $Y$) avoiding division of coefficients (which are polynomials in $X$), we do not expand the product of the coefficients of $\widetilde F$ times ${\rm lc}(F)^D$ at the beginning, and at the $i$th step, we write each coefficient of the remainder as a multiple of ${\rm lc}(F)^{D-i}$. Thus, at each step, we compute at most $d_Y + \delta_Y$ polynomials in $X$: for the first $d_Y$ of them, we compute the difference of two products of a coefficient of $F$ (whose degree is at most $d_X$) by a polynomial of degree bounded by $(i+1)d_X +\delta_X$, and for the other ones, the product of the leading coefficient of $F$ by a polynomial of degree bounded by $(i+1)d_X +\delta_X$. Then, the overall complexity of this step is $O((d_Y+\delta_Y) d_X\delta_Y(\delta_Y d_X+\delta_X))$.

\item[Step 2.]
Each subresultant of $F$ and $F_1$ is a polynomial in the variable $Y$ whose coefficients are polynomials of degree bounded by $(d_Y-1) d_X + d_Y ((\delta_Y+2) d_X +\delta_X)$ in the variable $X$. We compute it by means of interpolation: for sufficiently many interpolation points, we evaluate the coefficients of $F$ and $F_1$, we compute the corresponding determinant  (which is a polynomial in $Y$ with constant coefficients) and, finally we interpolate to obtain each coefficient.

For each interpolation point, the evaluation of the coefficients of $F$ and $F_1$ can be performed within complexity $O(d_Y d_X + (d_Y-1)((\delta_Y+2) d_X+\delta_X)) = O(d_Y(\delta_Yd_X+\delta_X))$.
Then, we compute at most $2d_Y-1$ determinants of matrices of size bounded by $2d_Y-2$ within complexity $O(d_Y^{\omega+1})$, we multiply them by the polynomials $Y^jF$ or $Y^jF_1$ evaluated at the point and we add the results in order to obtain the specialization of the subresultant at the point, which does not modify the complexity order.
This is repeated for $d_Y ((\delta_Y+3) d_X+\delta_X)$ points. Finally, each of the at most $d_Y$ coefficients of the subresultant polynomial is computed by interpolation from the results obtained.
Each polynomial interpolation can be done within complexity $O(M(d_Y ( \delta_Y d_X +\delta_X)) \log(d_Y ( \delta_Y d_X +\delta_X)))$.
Then, the computation of the at most $d_Y$ coefficients of each subresultant  can be achieved within complexity
$O((d_Y(\delta_Yd_X+\delta_X) + d_Y^{\omega+1})d_Y(\delta_Yd_X+\delta_X)+
 d_Y M(d_Y(\delta_Y d_X +\delta_X)) \log(d_Y(\delta_Y d_X +\delta_X))) = O (d_Y^{\omega+2}(\delta_Y d_X +\delta_X)^2) $.

 As we have to compute at most $d_Y$ subresultants, the overall complexity of the computation of all the required subresultants is of order
 $O(d_Y^{\omega+3}(\delta_Y d_X +\delta_X)^2)$.

 Note that we may compute successively only the polynomials $R_i = {\rm SRes}_{n_i-1}$. The index $n_{i+1}$ indicating the next subresultant to be computed is the degree of $R_i$, and the polynomial $\tau_i$ is its leading coefficient.
Finally, the polynomials $\rho_i \in \Z[X]$ are subresultant coefficients of $F$ and $F_1$, which are also computed by interpolation. The complexity of these computations does not modify the order of the overall complexity of this step.

\item[Step 3.] Consider the polynomial
\begin{equation}\label{poliL}
L(X) = \prod_{0\le i\le N} \ \tau_i \prod_{3\le i \le N+1} \rho_i.
\end{equation}
Note that $\rho_{2} = (-1)^{\frac{1}{2}(\deg_Y(F)-\deg_Y(F_1))(\deg_Y(F)-\deg_Y(F_1)+1)}{\rm lc}(F_1)^{\deg_Y(F)-\deg_Y(F_1)}$; so, it has the same zeros as  $\tau_{1} = {\rm lc}(F_1)$.

We determine the Thom encodings of the roots of $L$ in the interval $(a,b)$ by computing the realizable sign conditions on
$ {\rm  Der}(L), X-\alpha, \beta-X$, where ${\rm Der}(L) =( L, L',\dots, L^{\deg(L)})$.

The degree of $L$ is bounded by $(2d_Y^2-d_Y) ((\delta_Y+3)d_X+\delta_X)$.
We compute its  coefficients by interpolation: the specialization of $L$ at a point can be computed within $O(d_Y^2(\delta_Yd_X+\delta_X))$ operations by specializing its factors and multiplying, and this is done for $\deg(L)+1$ points; then, the total complexity of evaluation and interpolation is of order $O(d_Y^4(\delta_Yd_X+\delta_X)^2)$. The complexity of computing the realizable sign conditions is of order $O(d_Y^6 (\delta_Y d_X + \delta_X)^3 \log^3(d_Y^2(\delta_Y d_X + \delta_X)))$ (see Lemma \ref{perrucci11}). Finally, we can order the roots of $L$ in $(\alpha, \beta)$ by comparing their Thom encodings (see \cite[Proposition 2.28]{BPR}) within complexity $O(d_Y^4 (\delta_Y d_X+\delta_X)^2 \log (d_Y^2 (\delta_Y d_X+\delta_X)))$ using a sorting algorithm.

The overall complexity of this step is of order $O(d_Y^6 (\delta_Y d_X + \delta_X)^3 \log^3(d_Y^2(\delta_Y d_X + \delta_X)))$.

\item[Step 4.] The Sturm sequences $(\mathbf{f}_{I_j})_{0\le j \le k}$ are obtained by multiplying the polynomials $(R_i)_{0\le i \le N}$ by the corresponding signs $(\sigma_{I_j, i})_{0\le i \le N}$ as stated in Definition \ref{def:signseq}.
Note that if $p$ is a univariate polynomial having a constant sign in $I_j=(\alpha_j, \alpha_{j+1})$, to determine this sign it suffices to determine $\sg(p, \alpha_j^+)$ or $\sg(p, \alpha_{j+1}^-)$, which can be obtained from the signs of $p$ and its successive derivatives at $\alpha_j$ or $\alpha_{j+1}$ respectively.

Then, in order to compute the required signs, we compute the realizable sign conditions on the family
\[ \text{Der}(L),X-\alpha, \beta-X,  \text{Der}(\rho_{i})_{3\le i \le N},   \text{Der}(\tau_{i})_{1\le i \le N-1}\]
which consists of $O(d_Y^2(\delta_Y d_X + \delta_X))$ polynomials of degrees bounded by $(2d_Y^2-d_Y) ((\delta_Y+3)d_X+\delta_X)$. The complexity of this computation is of order $O(d_Y^6 (\delta_Y d_X+\delta_X)^3 \log^3 (d_Y^2 (\delta_Y d_X+\delta_X)))$.
Going through the list of realizable sign conditions, we determine the signs $\sigma_{I_j, i}$ and, from them, the Sturm sequences $\mathbf{f}_{I_j}$ within the same complexity order.

\medskip

The overall complexity of Steps 1 -- 4 is of order $O(d_Y^6 (\delta_Y d_X+\delta_X)^3 \log^3 (d_Y^2 (\delta_Y d_X+\delta_X)))$.

\item[Steps 5 and 6.] These steps require the determination of the sign of Pfaffian functions of the type  $G(x, \varphi(x))$, with $G\in \Z[X,Y] $, at real algebraic numbers given by their Thom encodings (more precisely, at the real roots  $\alpha_j$ of $L$ lying on $(\alpha, \beta)$ and at the endpoints $\alpha$ and $\beta$ of the given interval). We assume an oracle is given to achieve this task.

    At Step 5, we need  $k+2\le \deg(L)+2 =O(d_Y^2 (\delta_Yd_X+\delta_X))$ calls to the oracle for the Pfaffian function defined by the polynomial $F$, having degrees $\deg_X(F) = d_X$ and $\deg_Y(F)= d_Y$.

    At Step 6, we use the oracle for Pfaffian functions defined by polynomials with degrees in $X$ bounded by $d_Y((\delta_Y+3)d_X+\delta_X)$ and degrees in $Y$ bounded by $d_Y$. Taking into account the bound for the multiplicity of a zero of such a function given by Lemma \ref{multiplicity}, it follows that the determination of  $\sg(f_{I_j, i}, \alpha_\ell^+)$ and $\sg(f_{I_j, i}, \alpha_\ell^-)$ requires at most $O(d_Y(d_Y+\delta_Y) (\delta_Yd_X+\delta_X))$ calls to the oracle.
    Then, the oracle is used at most $O(d_Y^4(d_Y+\delta_Y) (\delta_Yd_X+\delta_X)^2)$ times.

\end{description}

Therefore, we have the following:

\begin{proposition}
Let $f(x) = F(x, \varphi(x))$  be defined from a polynomial $F\in \Z[X,Y]$ and a Pfaffian function $\varphi$ satisfying $\varphi'(x) = \Phi(x, \varphi(x))$, where $\Phi \in \Z[X,Y]$ with $\deg_Y(\Phi)>0$. Let $d_X:= \deg_X(F)$, $d_Y:=\deg_Y(F)$, $\delta_X:= \deg_X(\Phi)$ and $\delta_Y:= \deg_Y(\Phi)$. Then, Algorithm \texttt{ZeroCounting} computes the number of zeros of $f$ in a closed interval $[\alpha, \beta]\subset \mbox{Dom}(\varphi)$  ($\alpha, \beta \in \Q$) within  $O(d_Y^6 (\delta_Y d_X+\delta_X)^3 \log^3 (d_Y^2 (\delta_Y d_X+\delta_X)))$ arithmetic operations and comparisons, and using at most $O(d_Y^4(d_Y+\delta_Y) (\delta_Yd_X+\delta_X)^2)$ calls to an oracle for determining the signs of Pfaffian functions of the type $G(x, \varphi(x))$, with $G\in \Z[X, Y]$, at  real algebraic numbers.
\end{proposition}

As a consequence of the previous algorithm we deduce an upper bound for the number of zeros  of the Pfaffian functions under consideration in a bounded interval:

\begin{corollary} \label{coro:numberofzeros} Let $f(x) = F(x, \varphi(x))$  be defined from a polynomial $F\in \Z[X,Y]$ and a Pfaffian function $\varphi$ satisfying $\varphi'(x) = \Phi(x, \varphi(x))$, where $\Phi \in \Z[X,Y]$ with $\deg_Y(\Phi)>0$. Let $d_X:= \deg_X(F)$, $d_Y:=\deg_Y(F)$, $\delta_X:= \deg_X(\Phi)$ and $\delta_Y:= \deg_Y(\Phi)$.
Then, for any open interval $I\subset {\rm Dom}(\varphi)$, the number of zeros of $f$ in $I$ is at most $(d_Y +1) (2d_Y^2-d_Y) ((\delta_Y+3)d_X+\delta_X)$.
\end{corollary}

An alternative bound can be obtained from Khovanskii's upper bounds for the number of non-degenerate zeros of univariate Pfaffian functions and for the multiplicity of an arbitrary zero of these functions (see \cite{GV04}). Keeping our previous notation, for a polynomial $F\in \Z[X,Y]$ with $\deg(F) =d$, if $\deg(\Phi) = \delta$, using Khovanskii's bounds, it follows that both the number of non-degenerate zeros and the multiplicity of an arbitrary zero of $f(x) = F(x, \varphi(x))$ are at most $d (\delta +d)$. We can get an upper bound for the total number of zeros of $f$ by bounding the number of \emph{non-degenerate} zeros of $f$ and of its successive derivatives of order at most $d(\delta +d)-1$.

Following (\ref{Ptilde}), we have that $f'$ is defined by a polynomial of degree at most $d+\delta-1$ and so, for every $k\in \N$, $f^{(k)}$ is given by a polynomial of degree at most $d +k (\delta - 1)$. Then, the total number of zeros of $f$ is at most
$$\sum_{k=0}^{d(\delta+d)-1} (d+k(\delta-1))(\delta +d+k(\delta-1))\le \dfrac12 d^3 \delta^2 (\delta+d)^3.$$
Note that the bound from Corollary \ref{coro:numberofzeros} is of lower order than this one.

\section{E-polynomials} \label{sec:Epolynomials}

In this section, we will deal with the particular case of $E$-polynomials, namely when $\varphi(x) = e^{h(x)}$ for a polynomial $h\in \Z[X]$ of positive degree.
We will first show how to perform steps 5 and 6 of Algorithm  \texttt{ZeroCounting} (that is, we will give an algorithmic procedure to replace the calls to an oracle). Finally, we will prove a bound for the absolute value of the zeros of an $E$-polynomial.

\subsection{Sign determination}

The main goal of this section is to design a symbolic algorithm which determines the sign that an $E$-polynomial takes at a real algebraic number given by its Thom encoding.
To do this, we will use two subroutines. The first one, which follows \cite[Lemma 15]{Vor92}, determines the sign of an expression of the form $e^\beta- \alpha$ for real algebraic numbers $\alpha$ and $\beta$. The second one allows us to locate a real number of the form $e^{h(\alpha)}$, for a real algebraic number $\alpha$,  between two consecutive real roots of a given polynomial.

\medskip

\noindent \hrulefill

\smallskip

\noindent \textbf{Algorithm \texttt{SignExpAlg}}

\medskip

\noindent INPUT: Real algebraic numbers $\alpha$ and $\beta$ given by their Thom encodings $\sigma_{P_1} (\alpha)$ and $\sigma_{P_2} (\beta)$ with respect to polynomials $P_1, P_2 \in \Z[X]$ such that $\deg(P_1), \deg(P_2)\le d$ $(d\ge 2)$ and $H(P_1), H(P_2) \le H$.

\medskip

\noindent OUTPUT: The sign $s:= \textrm{sign}(e^\beta - \alpha)$.

\begin{enumerate}
\item Let $c:= (2^{d+1} (d+1)H)^{-2^{41} d^6 (5 d + 4\lceil \log(H) \rceil)}$.

\item Compute $w\in \Q$ such that $|e^\beta - w| <c$ as follows:
\begin{enumerate}
\item Compute $w_1\in \Q$ such that $|\beta -w_1|< \dfrac{c}{2.\, 3^{H+2}}$
\item Compute $w\in \Q$ such that  $|e^{w_1} - w| <\dfrac{c}{2}$
\end{enumerate}
\item Compute $s=\text{sign}(w-\alpha ) $.
\end{enumerate}
\noindent \hrulefill

\bigskip

\emph{Proof of correctness and complexity analysis:}

\begin{description}

\item[Step 1.] We will show that, for the chosen value of $c$, the inequality $|e^\beta - \alpha|>c$ holds.

As shown in \cite{Wald78}, if $\alpha$ and $\beta$ are algebraic numbers of degrees bounded by  $\theta$ and heights bounded by $\nu$, then
\[ |e^\beta - \alpha| > e^{-2^{42} \theta^6 \ln(\nu +e^e) (\ln (\nu)+\ln\ln(\nu))}\]

Note that
\[ e^{2^{42} \theta^6 \ln(\nu +e^e) (\ln (\nu)+\ln\ln(\nu))} \le (\nu +16 )^{2^{42} \theta^6 (\ln(\nu)+\ln\ln(\nu))}\le (\nu +16 )^{2^{43} \theta^6 \ln(\nu)}\]

It is clear that the degree of an algebraic number is bounded by the degree of any polynomial which vanishes at that number. With respect to the height, by \cite[Propositions 10.8 and 10.9]{BPR}, we have
\[ H(\alpha) \le 
 2^d ||P_1|| \le
2^d (d+1)^{1/2} H,\]
and, similarly, it follows that the same bound holds for $H(\beta)$. Here, $ ||P_1||$ stands for the norm $2$ of the vector of the coefficients of $P_1$.

The required inequality is deduced by taking  $\theta=d$, $\nu=2^d (d+1)^{1/2} H$, and using the bounds
\[2^d (d+1)^{1/2} H +16 \le 2^{d+1} (d+1) H \ \hbox { and } \ \ln(2^d (d+1)^{1/2} H) \le \dfrac{5}{4} d +\lceil \log(H)\rceil.\]

\item[Step 2.(a)] Applying the algorithm from Lemma \ref{intervalsforroots} to the polynomial $P_2$  with $\epsilon = \dfrac{c}{3^{H+3}}$, we get intervals $I_j=(a_j, b_j]$ with $a_j, b_j\in \Q$ and $b_j-a_j<\epsilon$ $(1\le j\le \kappa)$ such that $\beta \in I_{j_0}$ for some $j_0$. We determine the index $j_0$ by computing the feasible sign conditions for $\text{Der}(P_2), X-a_1, X-b_1,\dots, X-a_\kappa, X-b_\kappa$. Finally, we take $w_1 = b_{j_0}$. The complexity of this step is of order $O(d^3 (\log(H. 3^{H+3} .c^{-1}) + \log^3(d)))= O(d^3 H + d^{9}(d+\log(H))^2)$.

By the mean value theorem, the inequality $|\beta -w_1|< \dfrac{c}{2.\, 3^{H+2}}$ implies that $|e^\beta - e^{w_1}|<\dfrac{c}{2}$.

\item[Step 2.(b)] Following \cite[Lemma 14]{Vor92}, in order to obtain $w$, we compute the Taylor polynomial centered at $0$ of the function $e^x$ of order $t:= 8(\lceil \log(2/c)\rceil +1+H)$ specialized in $w_1$. The complexity of this step is bounded by $O(d^7(d+\log(H))^2 +H)$.

\item[Step 3.] The fact that $\text{sign}(w-\alpha)=\text{sign}(e^\beta - \alpha ) $ is  a consequence of the inequalities $|e^\beta -\alpha|>c$ and $|e^\beta -w|<c$.
 In order to determine this sign, we compute the feasible sign conditions on $\text{Der}(P_1), X-w$ and look for the one which corresponds to the Thom encoding of $\alpha$. The complexity of this step is of order $O(d^3 \log^3(d))$.

\end{description}

The overall complexity of this subroutine is $O(d^3 H + d^{9}(d+\log(H))^2)$.

\bigskip

The second subroutine is the following:

\noindent \hrulefill

\smallskip

\noindent \textbf{Algorithm \texttt{RootBox}}

\medskip

\noindent INPUT: A polynomial $h\in \Z[X]$, an algebraic number $\alpha\in \mathbb{R}$ such that $h(\alpha) \ne 0$,  given by its Thom encoding as a root of a polynomial $L\in \Z[X]$, and a polynomial $M\in \Z[X]$ together with the ordered list of Thom encodings of all its real roots $\lambda_1<\lambda_2<\dots < \lambda_m$.

\medskip

\noindent OUTPUT: The index $i_0$, $0\le i_0 \le m$, such that $\lambda_{i_0} < e^{h(\alpha)} <\lambda_{i_0+1}$, where $\lambda_0 = -\infty$ and $\lambda_{m+1} = +\infty$.

\begin{enumerate}
\item Compute $S(T):= \text{Res}_X(L(X), T- h(X))$.
\item Compute the feasible sign conditions on $\text{Der}(L), S(h), S'(h), \dots, S^{(\deg(S))}(h)$ and the Thom encoding of $h(\alpha)$ as a root of $S$.
\item Compute $\text{sign}(e^{h(\alpha)}- \lambda_i)$ applying Algorithm \texttt{SignExpAlg}, for $i=1,\dots, m$, until the first negative sign is obtained for $i_0$. If all the signs are positive, $i_0= m$.
\end{enumerate}

\noindent \hrulefill

\bigskip

\emph{Proof of correctness and complexity analysis:}

\medskip

Note that $h(\alpha)$ is a root of the polynomial $S\in \Z[T]$ computed in Step 1. Therefore, in Step 2, the sign condition on $\text{Der}(L), S(h), S'(h), \dots, S^{(\deg(S))}(h)$ having the Thom encoding of $\alpha$ as a root of $L$ in the first coordinates has the Thom encoding of $h(\alpha)$ as a root of $S$ in the last ones.

Assume that $\deg(L)\le \ell$, $\deg(h)\le \delta$ and $\deg(M)\le \eta$.

The resultant computation in Step 1 can be done within complexity $O(\ell(\ell +\delta)^{\omega})$ by interpolation, noticing that $\deg(S)\le \ell$.
Applying Lemma \ref{perrucci11}, the complexity of Step $2$ is $O(\ell^3 \delta\log(\ell)\log^2(\ell\delta))$. Finally, taking into account that $H(S) \le (\ell+\delta)!\,  H(L)^\delta (2H(h))^\ell $, defining
\[\mathcal{H}:= \max \{ H(M),  (\ell +\delta)!\,  H(L)^\delta (2H(h))^\ell \},\] the complexity of Step 3 is $O\Big(m \max\{\eta,\ell\}^3 \Big( \mathcal{H} +\max\{\eta,\ell\}^6 ( \max\{\eta,\ell\} + \log(\mathcal{H}))^2 \Big)\Big)$.

The overall complexity of the algorithm is of the same order as the complexity of Step 3.

\bigskip

Now we are ready to introduce the main algorithm of this section.

\noindent \hrulefill

\noindent \textbf{Algorithm \texttt{E-SignDetermination}}

\medskip
\noindent INPUT: Polynomials $G\in \Z[X,Y]$, $h\in \Z[X]$, $\deg(h)>0$, $L\in \Z[X]$ and Thom encodings $\sigma_L(\alpha_1), \dots, \sigma_L(\alpha_t)$ of real roots $\alpha_1,\dots, \alpha_t$ of $L$.

\medskip
\noindent OUTPUT: The signs of $G(\alpha_j, e^{h(\alpha_j)})$ for $1\le j \le t$.

\begin{enumerate}

\item For every $1\le j \le t$, determine whether $G(\alpha_j, Y) \equiv 0$. If this is the case, the sign of $G(\alpha_j, e^{h(\alpha_j)})$ is $0$.

\item Compute $R= {\rm gcd}(L,h)$ and the list of realizable sign conditions on $\textrm{Der}(L), R, G(X,1)$. Going through the list, determine the sign of $G(\alpha_j, e^{h(\alpha_j)}) = G(\alpha_j, 1)$ for every $j$ such that $G(\alpha_j, Y) \not \equiv 0$ and $R(\alpha_j)=0$.

\item Compute $M (Y):= \text{Res}_X(L(X), G(X,Y))$.

\item Compute the Thom encodings of the real roots of $M$ and order them: $\lambda_1<\dots< \lambda_m$.

\item For every $1\le j\le t$ such that $G(\alpha_j, Y) \not \equiv 0$ and $R(\alpha_j)\ne 0$:

\begin{enumerate}

\item Determine the index $0\le i_j\le m$ such that $\lambda_{i_j}< e^{h(\alpha_j)}< \lambda_{i_j+1}$ by applying subroutine \texttt{RootBox}, where $\lambda_0:=-\infty$ and $\lambda_{m+1}:=+\infty$.

\item Find $w_j \in \Q\cap (\lambda_{i_j}, \lambda_{i_j+1})$.

\item Compute the sign of the polynomial $G(X, w_j)$ at $X=\alpha_j$. This is the sign of $G(\alpha_j, e^{h(\alpha_j)})$.
\end{enumerate}

\end{enumerate}

\noindent \hrulefill

\bigskip

\emph{Proof of correctness and complexity analysis:}

\medskip

Assume that $\deg_X(G)\le d_X$, $\deg_Y(G)\le d_Y$, $\deg(L)\le \ell$ and $\deg(h)\le \delta$.

\smallskip

Due to Lindemann's theorem,
if $\alpha\in \mathbb{R}$ is an algebraic number and $h(\alpha) \ne 0$, then
$e^{h(\alpha)}$ is transcendental over $\Q$. Therefore, for an algebraic number
$\alpha\in \mathbb{R}$, $G(\alpha, e^{h(\alpha)}) = 0$  if and only if either
$G(\alpha, Y) \equiv 0$ or $h(\alpha)=0$ and $G(\alpha, 1) =0$.
Then, Steps 1 and 2 enable us to determine all the indices $j$ such that $G(\alpha_j, e^{h(\alpha_j)}) = 0$.

\begin{description}

\item[Step 1.] Compute $\textrm{cont}(G)$, the gcd of the coefficients of $G$, by applying successively the fast Euclidean algorithm \cite[Algorithm 11.4]{vzGG} within complexity $O(d_Y M(d_X) \log(d_X))$. Then, determine the realizable sign conditions on $\textrm{Der}(L), \textrm{cont}(G)$ within
$O(\ell^2 \max\{\ell, d_X\} \log(\ell)$ $ \log^2(\max\{\ell, d_X\}))$ arithmetic operations.

\item[Step 2.] The complexity of the computation of $R$ is of order $O(M(\max\{\ell, \delta\}) \log(\max\{\ell, \delta\}))$ and the realizable sign conditions on $\textrm{Der}(L), R, G(X,1)$ can be found within complexity $O(\ell^2 \max\{\ell, d_X\} \log(\ell) \log^2(\max\{\ell, d_X\}))$.

\item[Step 3.] In order to compute $M(Y)$, evaluate $G(X,y)$ at sufficiently many values $y$,
compute the corresponding determinants and interpolate.
Taking into account that $\deg(M)\le \ell d_Y$, the total cost of this step is of order $O(\ell d_Y (d_X +\ell)^\omega + M( \ell d_Y) \log(\ell d_Y))$.

\item[Step 4.] The computation of the Thom encodings of the real roots of $M$ can be done within $O((\ell d_Y)^3 \log^3(\ell d_Y))$ operations. Then, we order the real roots of $M$ by means of their Thom encodings within complexity of order $O((\ell d_Y)^2 \log(\ell d_Y))$.

\item[Step 5.] Following the proof of \cite[Proposition 8.15]{BPR}, it follows that $H(M) \le (\ell + d_X)! ((d_Y+1) H(G))^\ell H(L)^{d_X}$. Recall that $\deg(M)\le \ell d_Y$.

\begin{description}
\item[(a)] The complexity of this step is
$O((\ell d_Y)^4 (\mathcal{H} + (\ell d_Y)^6 ( \ell d_Y + \log(\mathcal{H}))^2 ) )$, where
$\mathcal{H} = \max \{ (\ell+\delta)! H(L)^{\delta} (2H(h))^\ell , (\ell+d_X)! H(L)^{d_X} ((d_Y+1)H(G))^\ell \}$.

\item[(b)] By applying Lemma \ref{intervalsforroots} to the polynomial $M$ and a lower bound $\epsilon$ for the minimum distance between two different roots of $M$, we obtain pairwise disjoint intervals $(a_i, b_i]$ with rational endpoints such that $\lambda_i\in (a_i, b_i]$ for $i=1,\dots, m$. Following Lemma \ref{separation}, we can take $\epsilon = (\ell d_Y)^{-\frac{\ell d_Y+2}{2}} (\ell d_Y+1)^{\frac{1-\ell d_Y}{2}}
((\ell+d_X)! \,  H(L)^{d_X} ((d_Y+1)H(G))^{\ell} )^{1-\ell d_Y}$.  Let $w_j := b_{i_j}$.

The complexity of this step is $O( (\ell d_Y)^4 ( (\ell +d_X) \log(\ell +d_X) + \ell (\log(H(G)) + \log(d_Y)) + d_X \log(H(L))))$.

\item[(c)] We compute the coefficients of $G(X, w_j)$ within complexity $O(d_X d_Y)$. Then, we compute the feasible sign conditions of $\text{Der}(L), G(X, w_j)$, which enable us to determine the sign of $G(\alpha_j, w_j)$, within $O(\ell^2 \max\{ \ell , d_X\} \log(\ell) \log^2(\max\{ \ell , d_X\})))$ additional operations.
\end{description}

\end{description}

The overall complexity of the algorithm is $O(t (\ell d_Y)^4 (\mathcal{H} + (\ell d_Y)^6
( \ell d_Y + \log(\mathcal{H}))^2 ))$.

\bigskip

The previous complexity analysis leads to:

\begin{proposition}
Given polynomials $G\in \Z[X,Y]$, $h\in \Z[X]$, $\deg(h)>0$, $L\in \Z[X]$ with degrees bounded by $d$ and height bounded by $H$, and Thom encodings $\sigma_L(\alpha_1), \dots, \sigma_L(\alpha_t)$ of real roots $\alpha_1,\dots, \alpha_t$ of $L$, we can determine $\#\{1\le  j\le t : G(\alpha_j, e^{h(\alpha_j)}) = 0\}$ within complexity $O(d^3 \log^3(d))$. Moreover, the signs of $G(\alpha_j, e^{h(\alpha_j)})$, for $1\le j \le t$, can be computed within complexity $O(t \, 8^{d} d^{3d+8} H^{2d})$.
\end{proposition}

\subsection{Zero counting for $E$-polynomials}

Here, we will apply Algorithm \texttt{E-SignDetermination} from the previous section as a subroutine in
Algorithm \texttt{ZeroCounting} described in Section \ref{sec:generalalgorithm}
to obtain a zero counting algorithm for
$E$-polynomials with no calls to oracles.

In order to estimate complexities we will need upper bounds for the degrees and heights of
polynomials defining the successive derivatives of an $E$-polynomial.

\begin{remark}\label{cotas}
For a Pfaffian function $g(x) = G(x, e^{h(x)})$, given by a polynomial $G\in\Z[X,Y]$,  we have that
$g'(x) = \widetilde G (x, e^{h(x)}) $
is given by the polynomial $\widetilde G:= \dfrac{\partial G}{\partial X} + h'(X) Y \dfrac{\partial G}{\partial Y}$.
If $\deg_X(G) = d_X$, $\deg_Y(G) =d_Y$ and $\deg(h) =\delta$, we have that
$$\deg_X(\widetilde G) \le \delta - 1+d_X  , \quad \deg_Y(\widetilde G) = d_Y $$
$$ H(\widetilde G) \le  H(G) (d_X + d_Y \delta^2 H(h))$$
Applying these bounds recursively, we get that the successive derivatives of $g$ can be obtained as
$$g^{(\nu)}(x) = {}^\nu\widetilde G (x, e^{h(x)})$$
for polynomials  ${}^\nu\widetilde G\in \Z[X,Y]$ such that
$$\deg_X({}^\nu\widetilde G) \le \nu(\delta - 1)+d_X , \quad \deg_Y({}^\nu\widetilde G) = d_Y $$
$$ H({}^\nu\widetilde G) \le  H(G) \prod_{j=0}^{\nu-1}(j(\delta - 1) +d_X+  d_Y \delta^2 H(h)).$$
\end{remark}

Now, we can state the main result of this section.

\begin{theorem}\label{thm:algcount}
Let $f(x) = F(x, e^{h(x)})$ be an $E$-polynomial defined by $F\in \Z[X,Y]$ and $h\in \Z[X]$ with
$\deg(F), \deg(h) \le d$ and $H(F), H(h) \le H$, and let $[a, b]$ be a closed interval. Assume that ${\rm{Res}}_Y(F, \widetilde F) \ne 0$. There is an algorithm that computes the number of zeros of $f$ in $[a, b]$ within complexity $(2dH)^{O(d^6)}$.
\end{theorem}

\begin{proof}{Proof.}
In order to prove the theorem, we adapt Algorithm \texttt{ZeroCounting} introduced in Section \ref{sec:generalalgorithm} to count the number of zeros of an $E$-polynomial with no call to oracles.
It suffices to show how to perform Steps 5 and 6 of the algorithm and estimate the complexity of the procedure.

Step 5 can be achieved by means of Steps 1 and 2 of Algorithm \texttt{E-SignDetermination}. As in this case $\deg(L)\le 10d^3$, the complexity is of order $O(d^9 \log^3(d))$.

To achieve Step 6 of the algorithm, we apply the algorithm \texttt{E-SignDetermination}
to the polynomials defining the functions $f_{I_j, i}$ and their successive derivatives, for $0\le i \le N$.
These functions are defined, up to signs, by the polynomials $R_i$ introduced in Notation \ref{not:subres}, and ${}^\nu\widetilde R_i$, $0\le i \le N,\ \nu\in \mathbb{N}$.

Since  $\deg_Y(\widetilde F)=\deg_Y(F)$, then
$F_1= {\rm lc}(F)^2 . \widetilde F - {\rm lc}(\widetilde F) {\rm lc}(F)F$ and so, $\deg_X(F_1)\le 4d-1$ and
$H(F_1) \le 4 d (d+1) H^3 (d+d^3 H)\le 8(d+1) d^4 H^4$.
Taking into account the determinantal formula for the subresultants, it follows that for every $k$,
$\deg_X({\rm SRes}_k) \le 5d^2-2d$ and
$H({\rm SRes}_k) \le (2d-1)! 2^{5d-2} (d+1)^{2d-2}  d^{5d-1} H^{5d-1}\le
3^{2d-1} 2^{5d-2} d^{9d-3} H^{5d-1}$, which are therefore, upper bounds for $\deg_X(R_{i})$ and
$H(R_{i})$ for all $i$.
Finally, recalling that $L$ is the product of at most $2d$ polynomials of degrees at most $5d^2-2d$ that are coefficients of the subresultants ${\rm SRes}_k$, we have that
$H(L) \le  (5d^2)^{2d-1} (3^{2d-1} 2^{5d-2} d^{9d-3} H^{5d-1})^{2d}
  \le  3^{4d^2} 2^{10d^2-2d} d^{18 d^2-2d-2} H^{10d^2 - 2d}$.

Taking into account the bound for the multiplicity of a zero of a Pfaffian function from
Lemma \ref{multiplicity}, we will apply the algorithm \texttt{E-SignDetermination} to the polynomials $R_i$  $(0\le i \le N)$ and ${}^\nu\widetilde R_i$ for $\nu\le 10d^3-3d^2$,
to determine the signs of the corresponding Pfaffian functions at the zeros of $L$.
The bounds from Remark \ref{cotas} applied to the polynomials $R_i$  imply that, for
$\nu\le 10d^3-3d^2$,
$$\deg_X({}^\nu\widetilde R_i) \le (10 d^3 -3d^2)(d-1) + 5d^2-2d\le 10 d^4-5d^3$$
$$ H({}^\nu\widetilde R_i) \le H(R_i) (10d^4 + (H-5) d^3)^{10d^3 -3d^2}$$
Then, the complexity of applying the algorithm to each of these polynomials is of order
$$O(d^{19} (\mathcal{H} + d^{24}(d^4 +\log \mathcal{H})^2))$$
where
$$\mathcal{H} \le  (10d^4+5d^3)! H(L)^{10d^4-5d^3} ((d+1)3^{2d-1} 2^{5d-2} d^{9d-3} H^{5d-1}(10d^4 + (H-5) d^3)^{10d^3 -3d^2})^{10d^3} $$
$$ = (2dH)^{O(d^6)}.$$

This sign computation is done for at most $d(10d^3-3d^2)$ polynomials. Finally, for each interval
$I_j$, the signs $\sg(f_{I_j, i}, \alpha_j^+)$ and $\sg(f_{I_j, i}, \alpha_{j+1}^-)$ are obtained easily
following Definition \ref{def:signseq}.

Therefore, the overall complexity of the algorithm is of order $$(2dH)^{O(d^6)}.$$
\end{proof}

\bigskip
The previous procedure can be slightly modified to count algorithmically the total number of real zeros of
an $E$-polynomial. To do this, we consider the signs of $E$-polynomials at $+\infty$ and $-\infty$.

Let $g(x) = G(x, e^{h(x)}) $ be an $E$-polynomial. Assume $G(X,Y) = \sum_{j=0}^{d_Y} a_j(X) Y^j$ with
$a_{d_Y}\ne 0$ and let $j_0 = \min\{ j : a_j \ne 0\}$. We define
\[ \sg(g, +\infty) = \begin{cases}
{\rm sign} ({\rm lc}(a_{j_0})) & {\rm if } \ {\rm lc}(h)< 0\\
{\rm sign} ({\rm lc}(a_{d_Y})) & {\rm if } \ {\rm lc}(h)> 0
\end{cases}\]
and
\[ \sg(g, -\infty) = \begin{cases}
{\rm sign} ((-1)^{\deg(a_{j_0})}{\rm lc}(a_{j_0})) & {\rm if } \ (-1)^{\deg(h)}{\rm lc}(h)< 0\\
{\rm sign} ((-1)^{\deg(a_{d_Y})}{\rm lc}(a_{d_Y})) & {\rm if } \ (-1)^{\deg(h)}{\rm lc}(h)> 0
\end{cases}\]

For a sequence of $E$-polynomials $\mathbf{f}= (f_0,\dots, f_N)$, we write $v(\mathbf{f}, +\infty)$ for the number of variations in sign  in $(\sg(f_0, +\infty), \dots, \sg(f_N, +\infty))$ and $v(\mathbf{f}, -\infty)$ for the number of variations in sign in $(\sg(f_0, -\infty), \dots, \sg(f_N, -\infty))$.

\begin{remark}
Following Notation \ref{not:subres} and Definition \ref{def:signseq},
let $\mathbf{f}_{I_{+\infty}}$ and $\mathbf{f}_{I_{-\infty}}$ be Sturm sequences for $f(x) = F(x, e^{h(x)})$ in the intervals $I_{+\infty} = (M, +\infty)$ and $I_{-\infty}= (-\infty, -M)$ where $M$ is an upper bound for the absolute values of the
roots of $\tau_i$ for $i=0,\dots, N$ and $\rho_i$ for $i=2,\dots, N+1$.

Then, the number of zeros of $f$ in $I_{+\infty}$ equals $v(\mathbf{f}, M^{+})-v(\mathbf{f}, +\infty)$ and the number of zeros of $f$ in $I_{-\infty}$  equals $v(\mathbf{f}, -\infty)-v(\mathbf{f}, -M^{-})$.
\end{remark}

By applying this remark, we conclude that the total number of zeros of an $E$-polynomial in $\mathbb{R}$
can be determined within the same complexity order as in Theorem \ref{thm:algcount}.

\begin{remark}
The assumption $\textrm{Res}_{Y}(F, \widetilde F) \ne 0$ in Theorem  \ref{thm:algcount} can be removed by using the construction in the proof of Lemma \ref{lem:resP}. Taking into account the increase of height and degree, it follows that the overall complexity of the root counting algorithm is of order $(2dH)^{d^{O(1)}}$ as stated in Theorem \ref{mainThm}.
\end{remark}

\subsection{Bound for the size of roots}

The following proposition provides an interval which contains all the zeros of an $E$-polynomial and
whose endpoints are determined by the degrees and heights of the polynomials involved in its definition. Using this bound, applying successively our algorithm for zero counting,
it is possible to separate and approximate the roots of an $E$-polynomial.

\begin{proposition}
Let $f(x) = F(x, e^{h(x)})$ be an $E$-polynomial defined by
$F\in \Z[X,Y]$ and $h\in \Z[X]$ such that $\deg(F)\le d $, $\deg (h) = \delta >0$ and
$H(F), H(h) \le H$.
Then, for every zero $\alpha \in \R$ of $f$, we have that $|\alpha | \le M(d, \delta, H):= 1+
(d+1) H^2 \max \{(d+1) (1+2H^2), \ 2 \lfloor \frac{2d}{\delta}+1\rfloor !  \}$.
\end{proposition}

\begin{proof}{Proof.}
 Let $F(X,Y) = \sum_{j=0}^{d_Y} a_j(X) Y^j\in \Z[X,Y]$ with
 $\deg(a_j) \le d_X$ for every $0\le j \le d_Y$ and
$a_{d_Y} \ne 0$.

Let $\alpha\in \R$ be a zero of $f$.  If $a_{d_Y} (\alpha) = 0$, then $|\alpha|\le r(a_{d_Y})< 1+H$ (see Lemma \ref{sizeofroots})
and so, the bound in the statement holds. Similarly, if $a_0(\alpha) = 0$, the bound holds.

Assume now that $a_{d_Y}(\alpha) \ne 0$ and $a_0(\alpha) \ne 0$. Then $e^{h(\alpha)} $ is a root of $F(\alpha, Y)\in \R[Y]$ and $e^{-h(\alpha)}$ is a root of $Y^{d_Y} F(\alpha, Y^{-1})\in \R[Y]$.
By Lemma \ref{sizeofroots}, it follows that
$$e^{2 h(\alpha)} <  1+ \sum_{0\le j \le d_Y-1} \left( \dfrac{a_j(\alpha)}{a_{d_Y}(\alpha)}\right)^2 \quad \hbox{ and } \quad e^{-2 h(\alpha)} <  1+ \sum_{1\le j \le d_Y} \left( \dfrac{a_j(\alpha)}{a_0(\alpha)}\right)^2 . $$

We are going to prove that, for $\alpha >M(d, \delta, H)$, one of the previous inequalities fails to hold.

Note that in both cases, the right hand side of the inequality is given by a rational function,
$$\dfrac{\sum_{0\le j \le d_Y} a_j(X)^2}{a_{d_Y}(X)^2} \qquad \hbox{ and } \qquad \dfrac{\sum_{0\le j \le d_Y} a_j(X)^2}{a_{0}(X)^2}$$
respectively, where the numerator and the denominator are integer polynomials of degrees at most $2d_X$ and coefficients of size bounded by $(d_Y+1) (d_X+1) H(F)^2$ and $(d_X+1) H(F)^2$ respectively. Moreover, the degree of the denominator is less than or equal to the degree of the numerator.

First, assume that the leading coefficient of $h$ is positive.

Let  $p(X) = \sum_{0\le j \le d_Y} a_j^2(X)$ and $q(X) = a_{d_Y}^2(X)$ so that $\frac{p(X)}{q(X)} =  1+ \sum_{0\le j \le d_Y-1} \left( \frac{a_j(X)}{a_{d_Y}(X)}\right)^2$.
and let $C>0$ be the quotient of the leading coefficients of $p$ and $q$. Note that $|C|\le (d_Y+1) H(F)^2$.

If $\deg(p) = \deg(q)$, for every $x> \max\{ r(q), r(p-(C+1)q)\}$, we have that $\dfrac{p(x)}{q(x)} < C+1$. On the other hand, for $x> r(2h-\ln(C+1))$, we have that $e^{2h(x)} >C+1$. We conclude that, for $x> \max \{ r(q), r(p-(C+1)q), r(2h-\ln(C+1)\}$, the inequality $e^{2h(x)}> \dfrac{p(x)}{q(x)}$ holds.

If $\deg(p) > \deg(q)$, let $d_0:=\deg(p)- \deg(q)$. For $x>\max\{ r(q), r(p-2C x^{d_0} q)\}$, we have that $\dfrac{p(x)}{q(x)}< 2C x^{d_0}$.
Note that $e^{2h(x)}> e^{x^\delta}$ for $x> r(2h-X^\delta)$. As $e^{x^\delta} > \sum_{k=0}^{\lfloor \frac{d_0}{\delta}+1\rfloor} \dfrac{1}{k!} x^{\delta  k} > 2Cx^{d_{0}}$ for $x> r( \sum_{k=0}^{\lfloor \frac{d_0}{\delta}+1\rfloor} \dfrac{1}{k!} X^{\delta  k} - 2C X^{d_0})$, it follows that $\dfrac{p(x)}{q(x)}< e^{2h(x)}  $ for $x> \max \{ r(q), r(p-2C x^{d_0} q),r( \sum_{k=0}^{\lfloor \frac{d_0}{\delta}+1\rfloor} \dfrac{1}{k!} X^{\delta  k} - 2C X^{d_0})\}$.
Using again Lemma \ref{sizeofroots}, we obtain:
\begin{itemize}
\item $r(q) < 1+(d_X+1) H(F)^2$
\item $r(p-(C+1)q) ) < 1+ (d_X+1)H(F)^2( d_Y+ (d_Y+1)H(F)^2)$
\item $r(2h - \ln(C+1)) < 1+ H(h) + \dfrac{1}{2} \ln( (d_Y+1)H(F)^2 +1)$
\item $r(p- 2C X^{d_0} q) < 1+ (d_X+1) (d_Y+1) H(F)^2 (1+2H(F)^2)$
\item $r(2h - X^\delta) < 1+ 2H(h)$
\item $r\Big(\sum_{k=0}^{\lfloor \frac{d_0}{\delta}+1\rfloor} \dfrac{1}{k!} X^{\delta  k} - 2C X^{d_0}\Big)< 1+ 2 \lfloor \frac{2d_X}{\delta}+1\rfloor !  (d_Y+1) H(F)^2$
\end{itemize}
and, therefore, we conclude that, for $\alpha >M(d,\delta, H)$, the following inequality holds
$$e^{2 h(\alpha)} >  1+ \sum_{0\le j \le d_Y-1} \left( \dfrac{a_j(\alpha)}{a_{d_Y}(\alpha)}\right)^2.$$

If the leading coefficient of $h$ is negative, applying the previous argument to $-h$, we have that, for $\alpha > M(d,\delta, H)$, the following  inequality holds
$$e^{-2h(\alpha)} > 1+ \sum_{1\le j \le d_Y} \left( \dfrac{a_j(\alpha)}{a_0(\alpha)}\right)^2.  $$

Finally, noticing that $\alpha$ is a zero of $F(x, e^{h(x)})$ if and only if $-\alpha$ is
a zero of $F(-x, e^{h(-x)})$  we conclude that every zero $\alpha$ of $f$ satisfies $\alpha \ge - M(d,\delta, H)$.
\end{proof}

\bigskip

\noindent \textbf{Acknowledgements.}  The authors wish to thank the referees for their detailed reading and helpful comments.

\end{document}